\documentclass[12pt]{amsart}

\usepackage{xcolor}
\usepackage[normalem]{ulem}
\usepackage[T1]{fontenc}

\usepackage[margin=2.5cm]{geometry}
\usepackage{amsthm}
\usepackage{enumerate}
\usepackage{amsmath}
\usepackage{hyperref}
\usepackage{dsfont} 

\theoremstyle{definition}
\newtheorem{definition}{Definition}
\newtheorem{theorem}[definition]{Theorem}
\newtheorem{lemma}[definition]{Lemma}

\newcommand{\mS}{\mathbb{S}^1}

\title[]{Ergodicity of $C^2$ minimal actions of Thompson group $T$ on the circle}
\author{Klaudiusz Czudek}
\address{Klaudiusz Czudek, Institute of Applied Mathematics, Faculty of Physics and Applied Mathematics, Gda{\'n}sk University of Technology, ul. Gabriela Narutowicza 11/12, 80-223 Gda{\'n}sk, Poland}
\email{klaudiusz.czudek@gmail.com}



\begin{document}

\begin{abstract}
We show that every $C^2$ minimal action of Thompson group $T$ on the circle is ergodic with respect to the Lebesgue measure. If such action is not minimal then the Lebesgue measure of the exceptional minimal set is zero.
\end{abstract}

\maketitle

\section{Introduction}
\subsection{Definitions and Main Theorem}
If $G$ is a group and $M$ a manifold, then a homeomorphism $\Phi : G \rightarrow \textrm{Homeo}_+(M)$ is called \textit{a representation} of $G$ in $\textrm{Homeo}_+(M)$. We say the representation $\Phi$ is $C^r$ (resp. $C^\omega$) if the image of $\Phi$ is contained in $\textrm{Diff}^r(M)$ (resp. $\textrm{Diff}^\omega(M)$). We usually identify $g\in G$ with $\Phi(g)$ and therefore assume that $G\subseteq \textrm{Homeo}_+(M)$ and say that $G$ \textit{acts} on $M$ via orientation preserving homeomorphisms. If $\mu$ is equipped with a reference measure $\mu$, then we say the representation is \textit{ergodic} with respect to $\mu$ if the condition $g(A)=A$ for every $g\in G$ and some measurable $A\subseteq \mathbb{S}^1$ implies $\mu(A)\in \{0, 1\}$. Group actions on manifolds play an important role in mathematics (see e.g. \cite{Brown_Fisher_Hurtado_20}, \cite{Brown_Fisher_Hurtado_22}, \cite{Fisher_20}, \cite{Navas_11}, \cite{Navas_18}, \cite{Mann_15}).

In this paper we are interested in group actions on the circle. In that case exactly one of the following cases holds (see Theorem 2.1.1 in \cite{Navas_11})
\begin{enumerate}[(i)]
\item there exists a point $x\in \mathbb{S}^1$ whose orbit is finite,
\item the orbit of an arbitrary point is dense (the action is minimal),
\item there exists a unique minimal invariant compact set homeomorphic to the Cantor set (moreover, this set is contained in the set of accumulation points of orbit of an arbitrary point).
\end{enumerate}

There are two still open major problem in this theory (Section 3.5 in \cite{Navas_11}). The first problem attributed to Ghys and Sullivan is if every $C^2$ minimal action on the circle is necessarily ergodic with respect to the Lebesgue measure. The second problem is if for any $C^2$ action admitting the exceptional minimal set, is it true that the Lebesgue measure of this set is necessarily 0?

The answers to both are affirmative in the case when the group has an element with the irrational rotation number (proven independently by Katok \cite{Katok_Hasselblatt_95}, Theorem 12.7.2, and Herman \cite{Herman_79}) or when for every point $x\in\mS$ there exists $g\in G$ such that $g'(x)>1$ (Shub and Sullivan \cite{Shub_Sullivan_85}). However, the latter condition is not always satisfied, thus for a given group $G$ it makes sense to define the set of nonexpandable points NE($G$)=$\{x\in \mS : g'(x)\le 1 \ \textrm{for every $g\in G$}\}$. Examples of actions with nonempty NE($G$) include the canonical action of PSL($2,\mathbb{Z}$) and smooth actions of Thompson group $T$. A crucial contribution in dealing with that case has been made by Deroin, Kleptsyn and Navas \cite{DKN_09}, who introduced so called conditions ($\ast$) and ($\Lambda\ast$) and proved that every $C^2$ minimal action with ($\ast$) is ergodic with respect to the Lebesgue measure and every $C^2$ action admitting the exceptional minimial set $\Lambda$ with ($\Lambda\ast$) has Leb($\Lambda$)=0. Those conditions are formulated as follows:

\begin{definition}
Let $G$ act on $\mS$ minimally. We say that $G$ has property ($\ast$) if for every  $x\in$ NE($G$) there exist $g_+$ and $g_-$ such that $g_+(x)=x$ and $g_+$ is parabolic repelling from the right, $g_-(x)=x$ and $g_-$ is parabolic repelling from the left.
\end{definition}

\begin{definition}
Let $G$ act on $\mS$ and admit the unique exceptional minimal set $\Lambda$. We say that $G$ has property ($\Lambda\ast$) if for every  $x\in$ NE($G$)$\cap \Lambda$ such that $x$ is an accumulation point of $\Lambda$ from the right there exist $g_+$ and such that $g_+(x)=x$ and $g_+$ is parabolic repelling from the right and for every  $x\in$ NE($G$)$\cap \Lambda$ such that $x$ is an accumulation point of $\Lambda$ from the left there exist $g_-$ such that $g_-(x)=x$ and $g_-$ is parabolic repelling from the left.
\end{definition}

Paper \cite{DKN_09} provides also nontrivial examples of actions with property ($\ast$). The actions with nonempty set of nonexpandable points has been studied subsequently (\cite{Kleptsyn_Filimonov_12'}, \cite{Kleptsyn_Filimonov_12}). The properties $(\ast)$ or $(\Lambda\ast)$ has been proven later for groups with one end satisfying some additional assumptions \cite{Filimonov_Kleptsyn_14}, for groups acting analytically on the circle that act freely \cite{DKN_18} or have infinitely many ends \cite{infinite_ends}.

One of the examples provided in \cite{DKN_09} with property ($\ast$) is the canonical action of PSL($2,\mathbb{Z}$). The other one are certain actions of Thompson group $T$, one of the famous Thompson group $F$, $T$ and $V$that are studied intensively nowadays (e.g. \cite{Poisson_boundary}, \cite{Thompson2}, \cite{Thompson3}, see also \cite{Cannon_Floyd_Parry_96} for a general introduction to Thompson groups). Thompson group $T$ can be defined as the group of all piecewise affine circle homeomorphisms with finite number of breaking points being diadic rational numbers and slopes being integer powers of two. Although this action is obviously not $C^1$, Thurston has shown that one can actually construct an action of $T$ on the circle by $C^\infty$ diffeomorphisms. Later the work of Ghys, Sergiescu \cite{Ghys_Sergiescu_87} has shown that there are plenty of possible actions of Thompson group $T$ on $\mS$ (this construction is described in Section \ref{S:Ghys}). For now we just state that to every circle covering map of degree two $\varphi$ with $\varphi(0)=0$ we can associate certain action of $T$ on the circle that is denoted in this paper by $T_\varphi$. Depending on the properties of $\varphi$ this actions is either minimal or admits the exceptional minimal set. The original action of $T$ by piecewise affine circle homeomorphisms is obtained when $\varphi(x)=2x$ mod $1$ and if $\varphi$ is $C^2$, $\varphi'(0)=1$, $\varphi''(0)=0$ then each element of $T_\varphi$ is a $C^2$ circle diffeomorphism. Deroin, Kleptsyn and Navas \cite{DKN_09} proved property ($\ast$) in the case when $\varphi'(x)>1$ for $x\not = 0$ (see more details in Section \ref{S:Ghys} and \cite{Navas_11}). Here we show that every minimal action of Thompson group $T$ on the circle satisfies ($\ast$). Moreover, \cite{DKN_09} do not provide any example of $C^2$ action admitting exceptional minimal set with property ($\Lambda\ast$).  Here we show that every such action of Thompson group $T$ satisfies ($\Lambda\ast$).

\begin{theorem}
\label{T:1}
Every $C^2$ minimal action of Thompson group $T$ on the circle satisfies property $(\ast)$ and therefore is ergodic. Every $C^2$ action admitting an exceptional minimal set satisfies property $(\Lambda\ast)$ and consequently the Lebesgue measure of the exceptional minimal set is zero.
\end{theorem}

The proof of the first part (minimal case) is an easy application of  Theorem 2.1 and Lemma 3.6 in \cite{LPS}. To show the second part (exceptional minimal set) we cannot apply those result directly. However, it is possible to modify the proofs to obtain the desired theorem. 

\subsection{Acknowledgments}
The author learned about the reference \cite{LPS} (that plays the fundamental role in the proof) from Sebastian van Strien, to whom he is grateful. The author is also thankful to Daniel Sell for useful discussions.

\section{Preliminaries}
\subsection{Notation}
If $G$ is a subgroup of Homeo$_+$($\mathbb{S}^1$), $x\in \mathbb{S}^1$, then the orbit of $x$ is the set Orb($x$)$=\{ g(x) : g\in G\}$. If $x, y\in \mS$, then $(x,y)$ stands for the open arch joining $x$ and $y$ in the counterclockwise orientation. $|J|$ is the length of the interval $J$.

\subsection{Ghys-Sergiescu smooth realization}
\label{S:Ghys}
(cf. Section 1.5.2 in \cite{Navas_11}, \cite{Ghys_Sergiescu_87} or Section 5.1 in \cite{DKN_09}). Throughout the paper we use the additive notation to describe the points on the circle. Let $\varphi : \mS \rightarrow \mS$ be the circle covering map of degree two with $\varphi(0)=0$. For each $n$ the set $\varphi^{-n}(0)$ consists of $2^n$ points that split $\mS$ into $2^n$ nonoverlapping closed intervals $I_{n,0},\cdots, I_{n,2^n-1}$. We call such intervals $\varphi$-dyadic intervals of degree $n$. We call an interval $I\subseteq \mS$ a $\varphi$-dyadic if it is $\varphi$-dyadic of degree $n$ for some $n$.

For a $\varphi$-dyadic interval $I$ of degree $n$ we denote by $\varphi_I$ the homeomorphism $\varphi_I: I \rightarrow [0,1]$, $\varphi_I:=\varphi^n$. Let now $T_\varphi$ be the set of orientation preserving circle homeomorphisms $g$ such that there exist two partitions of $\mS$ into $\varphi$-dyadic intervals $I_1,\cdots, I_r$ and $J_1,\cdots,J_r$ such that $g(I_k)=J_k$ and $g(x)=\varphi_{J_k}^{-1}\circ \varphi_{I_k}(x)$ for $x\in I_k$ for every $k=1,\cdots, r$. It can be show that $T_\varphi$ is closed under composition and forms a group that is isomorphic to the Thompson group $T$ (Section 1.5.2 in \cite{Navas_11}). It can be shown that every action of the Thompson group $T$ on the circle is of this form (Theorem 4.16 in \cite{Boudec_Bon_18}). Moreover, if $\varphi$ is a $C^r$ diffeomorphism locally in the neighborhood of every point, $\varphi'(0)=1$ and $\varphi^{(j)}(0)=0$ for $j=2,\cdots,r$ then every element of $T_\varphi$ is a $C^r$ diffeomorphism (Proposition 1.5.5 \cite{Navas_11}).

\subsection{Minimality and exceptional minimal set}\label{S:minimal}
Let $G$ be an arbitrary subgroup of Homeo$_+$($\mathbb{S}^1$). Theorem 2.1.1 in \cite{Navas_11} says that exactly one possibility occurs for the action of $G$:
\begin{enumerate}[(i)]
\item there exists a point $x\in \mathbb{S}^1$ whose orbit is finite,
\item the orbit of an arbitrary point is dense,
\item there exists a unique minimal invariant compact set homeomorphic to the Cantor set (moreover, this set is contained in the set of accumulation points of orbit of an arbitrary point) .
\end{enumerate}

Let now $G$ be the Thompson group $T_\varphi$ associated to the circle covering map $\varphi$. It is clear that the orbit of $x\in \mathbb{S}^1$ is the sum of forward and backward trajectory of $x$ by the action of $\varphi$. Since the backward trajectory in such case is always infinite, $T_\varphi$ does not admit a finite orbit. Further, by the Shub Theorem (Theorem 2.1 in \cite{deMelo_vanStrien}) $\varphi$ is semiconjugated to the doubling map $x\longmapsto 2x$ mod 1. If the semiconjugacy is a conjugacy then for every $y$ the set of points $x$ such that $\varphi^i(x)=y$ for some $i\ge 0$ is dense, which implies that all the orbits by the action of $G$ are dense and (ii) holds. If the semiconjugacy is not a conjugacy then, as explained in Remark 2 after Theorem 2.1 in \cite{deMelo_vanStrien}, there exist a finite number of maximal periodic (proper) intervals $I_1, \cdots, I_r$. It means there exist $s_1, \cdots, s_r$ such that $\varphi^{s_j}$ is a homeomorphism from $I_j$ onto itself for $j=1,\cdots, r$ and neither of $I_j$'s is contained in another proper interval with that property. Then
\begin{equation}
\label{E:2.lambda}
\Lambda = \mathbb{S}^1 \setminus \bigcup_{n=0}^\infty \varphi^{-n} (I_1\cup\cdots\cup I_r)
\end{equation}
is a compact, perfect and nowhere dense set and therefore is homeomorphic to the Cantor set. Moreover, it is full invariant, i.e. $\varphi^{-1}(\Lambda)=\Lambda$. The last property implies that $\Lambda$ is invariant for the group action and (iii) holds for $T_\varphi$. Therefore the action of Thompson group $T_\varphi$ on the circle is minimal if and only if $\varphi$ is conjugated to the doubling map, which holds if and only if $\varphi$ is topologically expanding (i.e. for every open interval $U$ there exists $n>0$ such that $\mathbb{S}^1\subseteq \varphi^n(U)$). By Corollary 2.2 in \cite{LPS} this is equivalent to the fact that for every periodic point $x$ of $\varphi$ with period $s$ there exists an open neighborhood $U$ of $x$ such that $|\varphi^s(y)-x|>|y-x|$ for every $y\in U$, $y\not = x$. When $\varphi$ satisfies the last condition we call it orbit expanding.

\begin{lemma}
\label{L:2.2}
Let $\varphi$ be a circle covering map of degree two. Let $\varphi$ be not topologically conjugated to the doubling map and let $\Lambda$ denote the full invariant set of $\varphi$ defined in \eqref{E:2.lambda}. Let $x$ be a periodic point of period $s$.
\begin{enumerate}[(a)]
\item If there exists $\varepsilon>0$ such that $(x, x+\varepsilon)\cap \Lambda$ is empty then $x$ is the left endpoint of a periodic interval of period $s$.
\item If for every $\varepsilon>0$ the set $(x, x+\varepsilon)\cap \Lambda$ is not empty then $\varphi^s$ is repelling is some right neighborhood of $x$. 
\end{enumerate}
The symmetric statement holds in the left neighborhood of $x$.
\end{lemma}

\begin{proof}
Let $I$ be the component of $\mS \setminus \Lambda$ whose left endpoint is $x$. Then according to \eqref{E:2.lambda} either $I$ is a periodic interval or $\varphi^n(I)$ is a periodic interval for some $n>0$. Since the left endpoint is periodic the latter cannot occur. This proves (a).

To show (2) it suffices to observe that if $\varphi^s$ is not repelling in some right neighborhood of $x$ then there exists another fixed point $y\in I$ of $\varphi^s$, where $I$ is the component of $\mS \setminus \varphi^{-s}(0)$ containing $x$.  Then the arc $(x, y)$ joining $x$ and $y$ in the couterclockwise orientation is clearly a periodic interval of period $s$ and, according to \eqref{E:2.lambda}, $\Lambda$ does not contain a point from $(x,y)$, which leads to a contradiction.
\end{proof}

\subsection{Nice intervals}

We call an interval $I$ nice if $\varphi^n(\partial I) \cap \textrm{int}(I)=\emptyset$ for every $n\ge 1$, where $\partial I$ and int$(I)$ stands for the boundary and interior of $I$, respectively. {

\begin{lemma}
\label{L:2.nice}Let $\varphi$ be a circle covering map of degree two. Let $\varphi$ be not topologically conjugated to the doubling map and let $\Lambda$ denote the full invariant set of $\varphi$ defined in \eqref{E:2.lambda}. Let $x_0\in \Lambda$ be a fixed point. If $(x_0, x_0+\varepsilon)\cap \Lambda$ is not empty for every $\varepsilon>0$, then there exists an arbitrarily small nice interval whose left endpoint is $x_0$. The analogous statement holds if $(x_0-\varepsilon, x_0)\cap \Lambda$ is not empty for every $\varepsilon>0$.
\end{lemma}
\begin{proof}
The set $\Lambda$ is nowhere dense, therefore our assumption implies that for every $\varepsilon>0$ there exists a component $I$ of $\mS \setminus \Lambda$ with $\overline{I}\subseteq (x_0, x_0+\varepsilon)$. By the definition of $\Lambda$, the left endpoint $z$ of $I$ is a preperiodic point, i.e. there exists a periodic point $p$ of period $s$ and a number $n$ such that $\varphi^n(z)=p$. Let $z'$ be  the element of $\bigcup_{j=0}^n \varphi^{-j}(\textrm{orb}(p))$ such that $(x_0,z')$ does not contain any other point from that set. Then $(x_0,z)$ is a nice interval.
\end{proof}

\subsection{Distortion}
Let
$$\chi(h, J) = \sup_{x,y\in J} \log \frac{h'(x)}{h'(y)}$$
be the distortion of  the function $h$ on interval J.

\begin{lemma}[cf. Lemma 3.1 in \cite{LPS}]
\label{L:3.1}
There exists a constant $C_0>0$ such that if $J$ is an interval and $n\ge 1$ is such that $J, \varphi(J), \cdots, \varphi^{n-1}(J)$ are pairwise disjoint, then
$$\chi(\varphi^n, J) \le C_0.$$
\end{lemma}

\section{Proof of Theorem \ref{T:1}}
Let $T_\varphi$ be the action of Thompson group $T$ on the circle associated to the circle covering map $\varphi$. By Theorem 4.14 \cite{Boudec_Bon_18} it is sufficient to consider only such actions.

\subsection{$G$ acts minimally on the circle}
If the action is minimal then $\varphi$ is necessarily orbit expanding as explained Section \ref{S:minimal}. By Theorem 2.1 in \cite{LPS} the set of parabolic periodic points Par($\varphi$) of $\varphi$ is finite and by Lemma 3.6 therein there exists $\lambda_0>1$ such that for every $x$ that is not a parabolic periodic point there exists $k \ge 1$ such that $(\varphi^k)'(x)\ge \lambda_0$, which readily implies that for every $x$ that is not a parabolic periodic point of $\varphi$ there exists $g\in T_\varphi$ such that $g'(x)>1$.

Let $x\in \textrm{NE}(T_\varphi)$. Then $x$ is a parabolic periodic point of $\varphi$ of some period $s$ and (as $\varphi$ is orbit expanding) $\varphi^s$ has no other fixed points in some $\varepsilon$-neighborhood of $x$. Clearly there exists $g\in T_\varphi$ with $g=\varphi^s$ on some neighborhood of $x$. Since $x\in \textrm{NE}(T_\varphi)$ was arbitrary, this completes the proof of property ($\ast$).

\subsection{There exists an exceptional minimal set $\Lambda$}
Let us now assume that $T_\varphi$ admits the exceptional minimal set $\Lambda$. The results from \cite{LPS} cannot be applied directly, nevertheless the proofs can be modified accordingly (see  Lemmata 3.4-3.6 \cite{LPS}).

\begin{lemma}[cf. Lemma 3.4 in \cite{LPS}]
\label{L:3.2}
Let $x_0$ be a fixed point such that $(x_0, x_0+\varepsilon)\cap \Lambda$ is not empty for every $\varepsilon>0$. Then for every $K\ge 1$ there exists $\varepsilon>0$ such that if $p\in \Lambda$ is a periodic point of period $s$ with $\varphi^i(p)\in (x_0, x_0+\varepsilon)$ for some $i=0,\cdots s-1$ then $(\varphi^s)'(p)\ge K$. The symmetric statement is true when $(x_0-\varepsilon, x_0)\cap\Lambda$ is not empty for every $\varepsilon>0$.
\end{lemma}

\begin{proof}[Proof of Lemma \ref{L:3.2}]
Let $x_0$ be a fixed point such that $(x_0, x_0+\varepsilon)\cap \Lambda$ is not empty for every $\varepsilon>0$. By Lemma \ref{L:2.2} $x_0$ is repelling from the right. Let $B_0$ be a nice interval whose left endpoint is $x_0$ (Lemma \ref{L:2.nice}). For each $n\ge 1$ define inductively $B_n$ to be the component of $\varphi^{-1}(B_{n-1})$ containing $x_0$. Then $B_n$ is a nice interval and $|B_n|\to 0$ as $n\to \infty$. Let
$$ \varepsilon_n = \sup \{ |J|: \textrm{$J$ is a component of $ \varphi^{-i}(B_n)$ for some $i\ge 0$} \}.$$

Since $B_{n+1}\subseteq B_n$ we get that $\varepsilon_n$ is nonincreasing. We are going to show that $\varepsilon_n \to 0$ as $n\to\infty$. Let us assume contrary to the claim that $\lim_{n\to\infty} \varepsilon_n=a>0$. There exist sequences $(J_n)$ and $(i_n)$ such that $J_n$ is a component of $\varphi^{-i_n}(B_n)$ such that $J_n, \varphi(J_n), \cdots, \varphi^{i_n-1}(J_n)$ are necessarily pairwise disjoint and $|J_n|\ge a$. {
}Using compactness we find a subsequence $(n_k)$ and an interval $J$ such that the left and right endpoints of $J_n$ converge to the left and right (respectively) endpoint of $J$. If $i_{n_k}\le M$ for some natural $M$ and every $k\ge 1$, then $|B_n|\ge \inf\{ |A|, |\varphi(A)|,\cdots, |\varphi^M(A)| \}$ for all sufficiently large $n$, where the infimum is taken over all intervals $A$ with $|A|\ge a/2$. The infimum is positive since $\varphi$ is a covering map, but on the other hand $|B_n|\to 0$, which leads to a contradiction. We thus conclude that $(i_{n_k})$ is an unbounded sequence. Let us recall that $J_n, \varphi(J_n), \cdots, \varphi^{i_n-1}(J_n)$ are necessarily pairwise disjoint, thus if $I$ is any interval such that $\overline{I}\subseteq J$ then $I, \varphi(I), \cdots, \varphi^{i_{n_k}-1}(I)$ are pairwise disjoint for sufficiently large $k$. Therefore $I$ is either asymptotic to a periodic orbit or a wandering interval. The first cannot occur but by Corollary 1 to Theorem I.2.2 in \cite{deMelo_vanStrien} there are no wandering intervals. This proves that $\varepsilon_n\to 0$. 

Let $\delta_0$ be the length of the components of $B_0\setminus B_1$. Choose $n$ so large that
$$\varepsilon_n\le \exp(-2C_0)\delta_0/K,$$
where $C_0$ is from Lemma \ref{L:3.1}. We are going to show that $B_n$ is the right neighborhood that satisfies the assertion, i.e. that $(\varphi^s)'(p)\ge K$ for any periodic point $p\in \Lambda$ of period $s$ such that $\varphi^i(p)\in B_n$. Since that derivative does not change when $p$ is replaced by any other point in the orbit of $p$, we can assume that $p$ is such that $(x_0, p)$ does not contain any other point from the orbit of $p$. Let $n_0$ be the unique $n_0\ge n$ such that $p\in B_{n_0}\setminus B_{n_0+1}$. Let $T=B_{n_0}\setminus B_{n_0+1}$ and let $J$ be the component of $\varphi^{-s}(B_{n_0})$ that contains $p$. The intervals $T, \varphi(T), \cdots, \varphi^{n_0}(T)$ are pairwise disjoint, $J\subseteq T$, thus $\chi(\varphi^{n_0}, T) \le C_0$ and $\chi(\varphi^{n_0}, J) \le C_0$ by Lemma \ref{L:3.1}. Since $s>n_0\ge n$, $\varphi^{n_0}(J)$ is a component of $\varphi^{n-s}(B_n)$, thus $|\varphi^{n_0}(J)|/|\varphi^{n_0}(T)| \le \varepsilon_n/\delta_0$ and by distortion estimate
$$\frac{|J|}{|T|} \le \exp(C_0) \frac{\varepsilon_n}{\delta_0}.$$

Using the above and Lemma \ref{L:3.1} again we obtain
$$(\varphi^s)'(p) \ge \exp(-C_0)\frac{|B_{n_0}|}{|J|} \ge \exp(-C_0)\frac{|T|}{|J|} \ge K,$$
which completes the proof.
\end{proof}

Let $\mathcal{L}$ (resp. $\mathcal{R}$) be the set of points that are a left (resp. right) endpoint of some component of $\mS\setminus \Lambda$.

\begin{lemma}
\label{L:3.epsilon}
Let $p_0\in \Lambda$ be a fixed point such that $(p_0, p_0+\varepsilon) \cap \Lambda$ is non empty for every $\varepsilon>0$. Let
$$
\varepsilon_n = \sup\{ |J|: \textrm{$J$ is a component of $\mS\setminus (\mathcal{L} \cup \varphi^{-n}(p_0) )$ whose left endpoint is not in $\mathcal{L}$ } \}.
$$
Then $\varepsilon_n\to 0$ as $n\to \infty$. The statement is still true when $(p_0-\varepsilon, p_0) \cap \Lambda$ is non empty for every $\varepsilon>0$, $\mathcal{L}$ is replaced by $\mathcal{R}$ and the phrase "left endpoint" in the definition of $\varepsilon_n$ is replaced by "right endpoint".
\end{lemma}
\begin{proof}
The sequence $\varepsilon_n$ is obviously nonincreasing, so let us assume contrary to the claim that $\varepsilon_n\to d>0$ as $n\to \infty$. There exist a sequence $(n_k)$ of natural numbers and a sequence of components $J_{n_k}=(a_{n_k}, b_{n_k})$ of $\mS\setminus(\varphi^{-n_k}(p_0) \cup \mathcal{L})$ such that $a_{n_k}\not \in \mathcal{L}$, $a_{n_k}\to a$, $b_{n_k}\to b$ and $|b-a|\ge d/2$. Since $J_{n_k}$ is a component of $\mS\setminus(\varphi^{-n_k}(p_0) \cup \mathcal{L})$, it cannot contain an endpoint of any of the intervals $J_{n_1},\cdots J_{n_{k-1}}$ and consequently we can assume that $(a_{n_k})$ is nondecreasing and $(b_{n_k})$ is nonincreasing, i.e. $J_{n_1}\supseteq J_{n_2} \supseteq \cdots$ The set $\Lambda$ is closed, $a_{n_k}\in \Lambda$ for $k\ge 1$, thus $a\in \Lambda$. The $\Lambda \cap (a,a+d/2)$ is empty, since $\bigcup_{n\ge 0} \varphi^{-n}(p_0)$ is dense in $\Lambda$, therefore $a\in \mathcal{L}$. But by the definition of $J_{n_k}$ this is possible only if $a_{n_k}=a$ for every $k$. This means $a_{n_k}\in \mathcal{L}$, which is again a contradiction with the choice of $J_{n_k}$. In the symmetric case the proof is analogous.
\end{proof}

\begin{lemma}[cf. Lemma 3.5 in \cite{LPS}]
\label{L:3.3}
For any $K\ge 1$ there exists $s_0$ such that if $p\in \Lambda$ is a periodic point of period $s \ge s_0$ then $(\varphi^s)'(p)\ge K$. In particular, the set of parabolic periodic points is finite.
\end{lemma}
\begin{proof}
Let $p_0\in \Lambda$ be a fixed point with $(p_0, p_0+\varepsilon)\cap \Lambda$ not empty for every $\varepsilon>0$. Let $\varepsilon_n$ be defined as in Lemma \ref{L:3.epsilon}, $n=1,2,\ldots$ By Lemma \ref{L:3.2} there exists $\varepsilon>0$ such that if $p\in \Lambda$ is a periodic point of period $s$ with $\varphi^i(p)\in (p_0, p_0+\varepsilon)$ for some $i=0,1,\cdots, s-1$ then $(\varphi^s)'(p)\ge K$. By Lemma \ref{L:3.epsilon} we can find $s_0$ so large that $\varepsilon_s \le \varepsilon/(e^{C_0}K)$ for $s\ge s_0$.

Let $p\in \Lambda$ be a periodic point with period $s\ge s_0$. If $\varphi^i(p)\in (p_0, p_0+\varepsilon)$ for some $i=0,1,\cdots, s-1$ then the assertion follows. If that is not the case, let $I$ be an open interval bounded by $p_0$ and some point $p'\in \textrm{orb}(p)$ with the property that $(p_0, p')$ does not contain any other point from orb($p$). Then $I$ is a nice interval and $|I|\ge \varepsilon$. Let $J$ be the component of $\varphi^{-s}(I)$ whose endpoint is $p'$. Since $\varphi^{-n}(p_0)$ is disjoint from $\mathcal{L}$ (see the definition just before Lemma \ref{L:3.epsilon}) $J$ is contained in a component of $\mS\setminus (\mathcal{L} \cup \varphi^{-n}(p_0) )$ whose left endpoint is not in $\mathcal{L}$ and thus $|J|\le \varepsilon_s$. By the choice of $s_0$, $|J| \le \varepsilon /(e^{C_0}K)$. Clearly $\varphi^i(J)\cap I=\emptyset$ for $i=1,2, \cdots, s-1$. By Lemma \ref{L:3.1} we have
$$(\varphi^s)'(p) = (\varphi^s)'(p') \ge e^{-C_0} \frac{|I|}{|J|}
\ge 
e^{-C_0} \frac{\delta}{\varepsilon_s} \ge K.
$$

From that is clear that the number of parabolic periodic points is finite.
\end{proof}

We are now in position to finish the proof. First we are going to show that if $x\in \Lambda$ is not a parabolic periodic point of $\varphi$ then there exists $s$ such that $(\varphi^s)'(x)>1$ (this part follows the lines of the proof of Lemma 3.6 \cite{LPS} with small changes). Then we are going to use that fact to show that $T_\varphi$ satisfies property $(\Lambda\ast)$.

By Lemma \ref{L:3.3} there exists $s_0$ such that if $p$ is a periodic point with period $s>s_0$ then $(\varphi^s)'(p)\ge 2\exp C_0$. Let $1<\lambda_0<\lambda_1<2$ be a constant such that if $p$ is periodic point of $\varphi$ that is not parabolic, $p\in \Lambda$, with period $s\le s_0$, then $(\varphi^s)'(p)>\lambda_1$ (this is possible as every periodic point in $\Lambda$ that is not parabolic is necessarily hyperbolic expanding). Let $\delta>0$ be such that $(\varpi^s)'(x)-(\varphi^s)'(y)|<\lambda_1-\lambda_0$ whenever $|x-y|<\delta$ and $s\le s_0$.

Let us recall that the full invariant set \eqref{E:2.lambda} is the set of points that do not fall eventually into the sum of a finite number of maximal periodic intervals. Let $\mathcal{P}\subseteq \Lambda$ be the set of points that are boundary points of this finite union. Let $x\in \Lambda$ be a point that is not parabolic periodic and let $A$ be a closed nice interval containing  $x$ with $|A|<\delta$. There exists such an interval since every component of $\mS\setminus \varphi^{-n}(\mathcal{P})$ is a nice interval and $\bigcup_{n>0}\varphi^{-n}(\mathcal{P})$ is dense in $\Lambda$.  By Lemma \ref{L:3.3} we can also assume that $A$ is disjoint from the set of parabolic periodic points.

Fix $y\in A$ and let $k$ be the first return time of $y$ to $A$. Let $J$ be the component of $\varphi^{-k}(A)$ which contains $y$. Then $\varphi^k$ is a diffeomorphism from $J$ to $A$ with distortion bounded by $C_0$. Since $J\subseteq A$, $\varphi^k$ contains a fixed point $p$ of $\varphi^k$. By the definition of $A$, $p$ is not a parabolic periodic point and $k$ is the fundamental period of $p$. If $k\le s_0$ then $(\varphi^k)(p)>\lambda_1>1$ by the choice of $\lambda_1$. Since $|x-p|<\delta$, $(\varphi^k)'(x)>(\varphi^k)'(p) - (\lambda_1-\lambda_0)>\lambda_0>1$ (by the choice of $\delta$). If $k>s_0$, then $(\varphi^k)'(p)\ge 2 \exp C_0$ by the choice of $s_0$. Thus by distortion estimates $(\varphi^k)'(x)\ge \exp(-C_0) (\varphi^^k)'(p)\ge 2$. This completes the proof of the first part.

Let $x\in \Lambda$. If $x$ is not a parabolic periodic point then $(\varphi^k)'(x)>1$ for some $k\ge 1$ by the first part of the proof. Since we can easily construct $g\in T_\varphi$ such that $g=\varphi^k$ on some neighborhood of $x$, thus $x$ is not in NE($T_\varphi$). Therefore if $x\in$NE($T_\varphi$), then $x$ is necessarily a parabolic periodic point of $\varphi$ with some period $s$. By Lemma \ref{L:2.2} (b) if $x$ is not isolated from $\Lambda$ from the right then $\varphi^s$ is necessarily parabolic repelling from the right. As before $g$ can be chosen so that $\varphi^s=g$ in some neighborhood of $x$. If $x$ is not isolated from the left then $g$ is constructed similarly. This proves property ($\Lambda\ast$).

\bibliographystyle{plain}
\bibliography{Bibliography}

\end{document}